\documentclass[12pt]{article}
\usepackage{times,amsmath,amsthm,amsfonts,amssymb}
\usepackage{hyperref}

\usepackage{pifont}
\newcommand{\cmark}{\ding{51}}
\newcommand{\xmark}{\ding{55}}

\newcommand{\iref}[1]{(\ref{#1})}
\newcommand{\be}{\begin{equation}}
\newcommand{\ee}{\end{equation}}
\newcommand\Chi{\raise .3ex\hbox{\large $\chi$}} 

\DeclareMathOperator{\argmin}{argmin}

\newtheorem{theorem}{Theorem}
\newtheorem{lemma}{Lemma}
\newtheorem{remark}{Remark}
\newtheorem{corollary}{Corollary}

\title{Optimal pointwise sampling for $L^2$ approximation}
\author{Matthieu Dolbeault\thanks{Sorbonne Universit\'e, UPMC Univ Paris 06, CNRS, UMR 7598, Laboratoire Jacques-Louis Lions, 4 place Jussieu, 75005 Paris, France (matthieu.dolbeault@sorbonne-universite.fr)} \ and Albert Cohen\thanks{Sorbonne Universit\'e, UPMC Univ Paris 06, CNRS, UMR 7598, Laboratoire Jacques-Louis Lions, 4 place Jussieu, 75005 Paris, France (albert.cohen@sorbonne-universite.fr)}}

\begin{document}
\maketitle

\begin{abstract}
Given a function $u\in L^2=L^2(D,\mu)$, where $\mu$ is a measure on a set $D$,
and a linear subspace $V_n\subset L^2$ of dimension $n$, we show that near-best approximation of $u$ in $V_n$ can be computed from
a near-optimal budget of $Cn$ pointwise evaluations of $u$, with $C>1$ a universal constant. The sampling points are drawn according to some random distribution, the approximation is computed by a weighted least-squares method, and the error is assessed in expected $L^2$ norm. This result improves on the results in \cite{CM,HNP} which require a sampling budget
that is sub-optimal by a logarithmic factor, thanks to a sparsification strategy introduced in \cite{MSS,NOU2016}. As a consequence, we obtain for any compact class $\mathcal K\subset L^2$ that
the sampling number $\rho_{Cn}^{\rm rand}(\mathcal K)_{L^2}$
in the randomized setting is dominated by the Kolmogorov $n$-width $d_n(\mathcal K)_{L^2}$. While our result shows the existence of a randomized sampling
with such near-optimal properties, we discuss remaining issues
concerning its generation by a computationally efficient algorithm.

\end{abstract}

\section{Introduction}

We study the approximation of a function $u\in L^2(D,\mu)$, where $\mu$ is a measure on a set $D$, by an element $\tilde u$ of $V_n$, a subspace of $L^2(D,\mu)$ of finite dimension $n$, based on pointwise data of
of $u$. Therefore, to construct $\tilde u$, we are allowed to evaluate $u$ on a sample of $m$ points $X=\{x^1,\dots,x^m\}\in D^m$. In addition, we consider randomized sampling and reconstruction, in the sense that $X$ will be drawn according to a distribution $\sigma$ over $D^m$, so the error $u-\tilde u$ should be evaluated in some probabilistic sense.
For the sake of notational simplicity, having fixed $D$ and $\mu$, we write
throughout the paper
\be
L^2:=L^2(D,\mu)\quad {\rm and}\quad \|v\|:=\|v\|_{L^2}=\left(\int_D |v|^2 d\mu\right)^{1/2},
\ee
as well as
\be
e_n(u):=\min_{v\in V_n}\|u-v\|.
\ee
One typical applicative setting is the reconstrution of multivariate functions, which corresponds to $D$ being a domain in $\mathbb R^d$.

Our main result is the following:

\begin{theorem}
\label{main theorem}
For some universal constants $C, K\geqslant1$, and for any $n$-dimensional space $V_n\subset L^2$, there exists a random sampling $X=\{x^1,\dots,x^m\}$ with $m\leqslant C n$ and a reconstruction map $R: D^m\times \mathbb C^m\mapsto V_n$, such that for any $u\in  L^2$,
\be
\mathbb E_X\left(\|u-\tilde u\|^2\right)\leqslant K e_n(u)^2
\label{mainest}
\ee
where $\tilde u:=R(x^1,\dots,x^m,u(x^1),\dots,u(x^m))$.
\end{theorem}

The reconstruction map $R$ is obtained through a weighted least-squares method introduced in \cite{CM}, which has already been discussed in several papers, see \cite{AC, CD, HNP, JNZ, DH, Mig2, Mig}. The weights involved are given by the expression
\begin{equation}
w:x\in D\mapsto n\min_{v\in V_n}\frac{\|v\|^2}{|v(x)|^2}=\frac{n}{\sum_{j=1}^n |L_j(x)|^2},
\label{weight function}
\end{equation}
where the last formula holds for any $L^2$-orthonormal basis $(L_1,\dots,L_n)$ of $V_n$,
which, up to the factor $n$, is the {\it Christoffel function} associated
to the space $V_n$ and the space $L^2(D,\mu)$.
The weighted least-squares solution is then simply defined as
\be
\tilde u:= \underset{v\in V_n}{\argmin}\sum_{i=1}^m w(x^i)|u(x^i)-v(x^i)|^2.
\ee
Introducing the discrete $\ell^2$ norm
\be
\|v\|_X^2:=\frac{1}{m}\sum_{i=1}^m w(x^i)|v(x^i)|^2
\ee
and its associated scalar product $\langle\cdot,\cdot\rangle_X$, we get a computable formula for $\tilde u$:
\be
\tilde u= \underset{v\in V_n}{\argmin}\|u-v\|_X^2=P^X_{V_n}u,
\ee
where $P^X_{V_n}$ denotes the orthogonal projection on $V_n$ with respect to $\langle\cdot,\cdot\rangle_X$. Note that strictly speaking $\|\cdot\|_X$ is not a norm over $L^2$, however
the existence and uniqueness of $P^X_{V_n}$ will be ensured by the 
second condition in Lemma \ref{main lemma} below, see Remark \ref{remgram}.

Therefore our main achievements lie in the particular choice of the random sample $X$
for ensuring the near-optimal approximation and sampling budget in Theorem \ref{main theorem}.

Now, the proof of Theorem \ref{main theorem} relies on two conditions: first, the expectation of $\|\cdot\|_X^2$ has to be bounded by $\|\cdot\|^2$ up to a constant. Second, an inverse bound should hold almost surely, instead of just in expectation, for functions $v$ in $V_n$. More precisely, one has:

\begin{lemma}
\label{main lemma}
Assume that $m$ and the law $\sigma$ of $X=\{x^1,\dots,x^m\}$ are such that
\begin{equation}
\mathbb E(\|v\|_X^2)\leqslant \alpha \|v\|^2, \qquad v\in L^2,
\label{condition on expectancy}
\end{equation}
and
\begin{equation}
\|v\|^2 \leqslant \beta \|v\|_X^2\;\;\text{ a.s.},\qquad v\in V_n.
\label{condition in V_n}
\end{equation}
Then
\be
\mathbb E(\|u-\tilde u\|^2)\leqslant (1+\alpha\beta) e_n(u)^2.
\label{statement}
\ee
\end{lemma}

\begin{proof}
Denote $u_n$ the orthogonal projection of $u$ on $V_n$ with respect to the $L^2(D,\mu)$ norm. Applying Pythagoras theorem both for $\|\cdot\|$ and $\|\cdot\|_X$, one obtains
\begin{align*}
\mathbb E(\|u-\tilde u\|^2)&=\|u-u_n\|^2+\mathbb E(\|u_n-\tilde u\|^2)\\
&\leqslant \|u-u_n\|^2+\beta \mathbb E(\|u_n-\tilde u\|_X^2) \\
&=\|u-u_n\|^2+\beta \mathbb E(\|u_n-u\|_X^2 -\|u-\tilde u\|_X^2)\\
&\leqslant\|u-u_n\|^2+\beta \mathbb E(\|u_n-u\|_X^2)\\
&\leqslant (1+\alpha \beta)\|u-u_n\|^2,
\end{align*}
which proves \iref{statement} since $\|u-u_n\|=e_n(u)$.
\end{proof}

In section 2, we recall how both conditions \iref{condition on expectancy} and \iref{condition in V_n} can be obtained with $m$ quasi-linear in $n$, that is, of order $n\log n$. We reduce this budget to $m$ of order $n$ in section 3, by randomly subsampling the set of evaluation points, based on results from \cite{MSS,NOU2016}. The proof of Theorem \ref{main theorem} follows. We compare it to the recent results \cite{KU2, LT, NSU} in section 4, in particular regarding the domination of sampling numbers by $n$-widths. We conclude in section 5 by a discussion on the {\it offline} computational cost for practically 
generating the sample $X$.

\section{Weighted least-squares}

A first approach consists in drawing the $x^i$ independently according to the same distribution $\rho$, that is, taking $\sigma=\rho^{\otimes m}$. The natural choice for $\rho$ is $d\rho=\frac{1}{w}d\mu$, which is a probability measure since
\be
\int_D \frac{1}{w}d\mu=\frac{1}{n}\sum_{j=1}^n\int_D |L_j(x)|^2d\mu(x)=\frac{1}{n}\sum_{j=1}^n\|L_j\|^2=1.
\ee
We denote by $Z=\{x^1,\dots,x^m\}$ this first random sample and by $\|\cdot\|_Z$ the
corresponding discrete $\ell^2$ norm. With this sampling measure,
\be
\mathbb E(\|v\|_Z^2)=\frac{1}{m}\sum_{i=1}^m\int_D w(x)|v(x)|^2d\rho=\int_D |v|^2d\mu=\|v\|^2,
\ee
so condition \iref{condition on expectancy} is ensured for $X=Z$ with $\alpha=1$. To study the second condition, we introduce the Hermitian positive semi-definite Gram matrix 
\be
G_Z:=(\langle L_j,L_k\rangle_Z)_{ j, k=1,\dots, n}
\ee
and notice that \iref{condition in V_n} is equivalent to
\be
 |\nu|^2=\Big\|\sum_{j=1}^n\nu_jL_j\Big\|^2\leqslant \beta\Big\|\sum_{j=1}^n\nu_jL_j\Big\|_Z^2=\beta \nu^* G_Z\nu,\qquad \nu\in \mathbb C^n,
\ee
which in turn rewrites as $\lambda_{\min}(G_Z)\geqslant \beta^{-1}$. 

By the central limit theorem, as $m$ tends to infinity, the scalar products $\langle L_j,L_k\rangle_Z$ converge almost surely to $\langle L_j,L_k\rangle=\delta_{j,k}$, so $G_Z$ converges to the identity matrix, and
we expect that $\lambda_{\min}(G_Z)\geqslant \beta^{-1}$ holds for $\beta>1$ with high probability
as $m$ gets large. A quantitative formulation can be obtained by studying the concentration of $G_Z$ around $I$ in the matrix spectral norm
$$
\|M\|_2:=\max\{|Mx| \,: \, |x|=1\}.
$$
This is based on the matrix Chernoff bound, see \cite{AW, Tr} for the original inequality and \cite{CD}, Lemma 2.1, for its application to our problem:

\begin{lemma}
\label{chernoff}
For $m\geqslant 10n\ln (\frac{2n}\varepsilon)$, if $X\sim \rho^{\otimes m}$, then
\be
\mathbb P\left(\|G_Z-I\|_2\leqslant\frac{1}{2}\right)\geqslant 1-\varepsilon.
\ee
In particular, $\mathbb P\left(\lambda_{\min}(G_Z)\geqslant\frac{1}{2}\right)\geqslant 1-\varepsilon$.
\end{lemma}

Thus assumption \iref{condition in V_n} is satisfied with $\beta=2$, but only with probability $1-\varepsilon$. As we would like it to hold almost surely, we condition the sampling to the event
\be
E:=\left\{\|G_Z-I\|_2\leqslant\frac{1}{2}\right\}
\ee
which defines a new sample
\be
Y=Z|E.
\ee
In practice, $Y$ can be obtained through a rejection method, which consists in
drawing successively sets of points $Z^1,Z^2,\dots$ according to $\rho^{\otimes m}$, 
and defining $Y=Z^k$ for the first value $k$ such that $E$ holds. 
We then define $\tilde u$ as the weighted least-square estimator based on this 
conditioned sample, that is
\be
\tilde u:=P_{V_n}^Y u.
\ee
This approach was introduced
and analyzed in \cite{HNP}, see in particular Theorem 3.6 therein.
A simpler version of this result, sufficient for our purposes, is the following:

\begin{lemma}
\label{boosted}
For $m\geqslant 10n\ln (4n)$, if $Z\sim\rho^{\otimes m}$ and $Y=Z|E$, then
\be
\|G_Y-I\|_2\leqslant\frac{1}{2},
\ee
and
 \be
 \mathbb E_Y(\|u-\tilde u\|^2)\leqslant 5e_n(u)^2.
 \ee
\end{lemma}

\begin{proof}
The first part immediately results from the definition of $Y$ and $E$, and implies condition \iref{condition in V_n} with $\beta=2$. Moreover, $\mathbb P(E)\geqslant\frac{1}{2}$ by Lemma \ref{chernoff}
with $\varepsilon=\frac 1 2$, so for any $v\in L^2(D,\mu)$,
\be
\mathbb E_Y(\|v\|_Y^2)=\mathbb E_Z(\|v\|_Z^2|E)=\frac{\mathbb E_Z(\|v\|_Z^2\Chi_E)}{\mathbb P(E)}\leqslant \frac{\mathbb E_Z(\|v\|_Z^2)}{\mathbb P(E)}\leqslant 2\|v\|^2,
\ee
so condition \iref{condition on expectancy} holds with $\alpha=2$. The conclusion follows from Lemma \ref{main lemma}.
\end{proof}

\begin{remark}
\label{remark redraws}
The number of redraws $k$ for reaching $Y$ follows a geometric law of expectation
$\mathbb E(k)=\mathbb P(E)^{-1}=(1-\varepsilon)^{-1}$, that is $\mathbb E(k)\leqslant 2$ for the particular choice of $m$ in the above lemma.
It should be well noted that $u$ is not evaluated at the 
intermediately generated samples $Z^1,\dots,Z^{k-1}$, which thus enter the offline cost
of the sampling algorithm.
\end{remark}

\begin{remark}
\label{remgram}
The fact that the Gramian $G_Y$ is non-singular implies that for any $u$ with given
values $y^i$ at the points $x^i$, we can uniquely define 
\be
\tilde u=P^Y_{V_n}u=\sum_{j=1}^n a_jL_j,
\ee
since $a=(a_1,\dots,a_n)^*$ solves the system of normal equations
\be
G_Y a=b, 
\ee
where the right-side vector has coordinates
\be
b_j=\langle L_j,u\rangle_Y=\frac 1 m\sum_{i=1}^m w(x^i)L_j(x^i)\overline{y^i}.
\ee
If $u$ is 
in $L^2$, the $y^i$ are only defined up to a representer, however since two representers
$u^1$ and $u^2$ coincide $\mu$-almost surely, we find that $P^Y_{V_n}u$ is 
well defined almost surely over the draw of $Y$.
\end{remark}

\section{Random subsampling}

With Lemma \ref{boosted}, we already have an error bound similar to that of Theorem \ref{main theorem}. However, the sampling budget is larger than $n$ by a logarithmic factor, which we seek to remove in this section. To do so, we partition the sample $Y$ into subsets of size comparable to $n$, and randomly pick one of these subsets to define the new sample. An appropriate choice of the partitioning is needed to circumvent the main obstacle, namely the preservation of condition \iref{condition in V_n}. It relies on the following lemma, taken from Corollary B of \cite{NOU2016}, itself a consequence of Corollary 1.5 in \cite{MSS}. The relevance of these two results to sampling problems 
were exploited in \cite{NSU} and noticed in \cite{KUV}, respectively.

\begin{lemma}
\label{MSS}
Let $a_1,\dots,a_m\in\mathbb C^n$ be vectors of norm $|a_i|^2\leqslant \delta$ for $i=1,\dots,m$, satisfying
\be
\alpha I\leqslant \sum_{i=1}^m a_ia_i^*\leqslant \beta I
\ee
for some constants $\delta < \alpha \leqslant \beta$. Then there exists a partition of $\{1,\dots,m\}$ into two sets $S_1$ and $S_2$ such that
\be
\frac{1-5\sqrt{\delta/\alpha}}{2}\alpha I\leqslant \sum_{i\in S_j}a_ia_i^*\leqslant \frac{1+5\sqrt{\delta/\alpha}}{2}\beta I,\qquad j=1,2.
\ee
\end{lemma}

In Lemma 2 of \cite{NOU2016} this result is applied inductively in order to find 
a smaller set $J\subset \{1,\dots,m\}$ of cardinality $|J|\leqslant cn$ such that 
\be
C^{-1} I\leqslant \frac m n\sum_{i\in J}a_ia_i^*\leqslant C I,
\ee
for some universal constants $c,C>1$. We adapt this approach in order to obtain
a complete partition of $\{1,\dots,m\}$ by sets having such properties.

\begin{lemma}
\label{induction}
Let $a_1,\dots,a_m\in\mathbb C^n$ be vectors of norm $|a_i|^2=\frac n m$ for $i=1,\dots,m$, satisfying
\be
\frac{1}{2} I\leqslant\sum_{i=1}^m a_ia_i^*\leqslant \frac{3}{2} I.
\ee
Then there exists an integer $L$ and a partition of $\{1,\dots,m\}$ into $2^L$ sets $J_1,\dots,J_{2^L}$ such that
\be
c_0 I\leqslant \frac{m}{n}\sum_{i\in J_k}a_ia_i^*\leqslant C_0 I,\qquad 1\leqslant k \leqslant 2^L,
\label{stabJ}
\ee
with universal constants $c_0$ and $C_0$. In addition, each set $J_k$ satisfies
\be
|J_k|\leqslant C_0 n.
\label{cardJ}
\ee
\end{lemma}

\begin{proof}
The cardinality estimate \iref{cardJ} follows from the upper inequality in \iref{stabJ} by taking
the trace 
\be
nC_0={\rm tr}(C_0I) \geqslant \frac{m}{n}\sum_{i\in J_k}{\rm tr}(a_ia_i^*)
=\frac{m}{n}|J_k|\frac{n}{m}=|J_k|.
\ee
For the proof of \iref{stabJ}, if $n/m \geqslant 1/200$, then the result holds with $L=0$, $J_1 = \{1,\dots,m\}$, $c_0=1/2$ and $C_0 = 300$. Now assuming $\delta := n/m < 1/200$, define by induction $\alpha_0=\frac{1}{2}$, $\beta_0=\frac{3}{2}$, and
\be
\alpha_{\ell+1} :=\alpha_\ell \frac{1-5\sqrt{ \delta/\alpha_\ell}}2, \qquad \beta_{\ell+1} :=\beta_\ell \frac{1+5\sqrt{ \delta/\alpha_\ell}}2, \qquad \ell\geqslant0.
\ee
As $\alpha_{\ell+1}\leqslant \frac{\alpha_\ell}{2}$, the minimal integer $L$ such that $\alpha_L\leqslant 100\delta$ is well defined, and satisfies
\be
\alpha_L=\alpha_{L-1} \frac{1-5\sqrt{ \delta/\alpha_{L-1}}}2> 100\delta \frac{1-5\sqrt{ 1/100}}2=25\delta.
\ee 
Moreover $\alpha_\ell\geqslant2^{L-\ell-1}\alpha_{L-1}\geqslant 2^{L-\ell-1}\,100\delta$ for $\ell=0,\dots,L-1$,
so
\be
\beta_L=3\alpha_L\prod_{\ell=0}^{L-1}\frac{1+5\sqrt{\delta/\alpha_\ell}}{1-5\sqrt{\delta/\alpha_\ell}}\leqslant C\delta,
\ee
with $C:=300\prod_{\ell\geqslant2}\frac{1+\sqrt 2^{-\ell}}{1-\sqrt 2^{-\ell}}$.

Finally, we inductively define partitions $\{S^\ell_1,\dots,S^{\ell}_{2^\ell}\}$ for $0\leqslant\ell\leqslant L$:
start with $S^0_1=\{1,\dots,m\}$ and for any $\ell,j$, noticing that
\be
\alpha_\ell I\leqslant \sum_{i\in S^\ell_j}a_ia_i^*\leqslant \beta_\ell I,
\ee
apply Lemma \ref{MSS} to split $S^\ell_j$ into subsets $S^{\ell+1}_{2j-1}$ and $S^{\ell+1}_{2j}$ satisfying the same property. 
At the last step, we define
\be
J_k=S^{L}_k.
\ee
The framing \iref{stabJ} thus holds with $c_0=\alpha_L/\delta\geqslant 25$ and $C_0=\beta_L/\delta \leqslant 11000$.
\end{proof}
\noindent
{\bf Proof of Theorem \ref{main theorem}:}
Define
\be
a_i=\left(\sqrt{\frac{w(x^i)}{m}}L_j(x^i)\right)_{j=1,\dots,n}
\ee
the normalised random vectors corresponding to the 
sample $Y=\{x^1,\dots,x^m\}$ introduced in the previous section. As
\be
\frac{1}{2}I\leqslant G_Y=\sum_{i=1}^ma_ia_i^*\leqslant \frac 3 2 I
\ee
and
\be
|a_i|^2=\frac 1 m w(x^i)\sum_{j=1}^n|L_j(x^i)|^2= \frac n m
\ee
thanks to the choice of weights \iref{weight function},
the assumptions of Lemma \ref{induction} are satisfied. Applying this lemma,
we obtain sets $J_1,\dots,J_{2^L}$ partitioning $\{1,\dots,m\}$. Let $\kappa$ be a random variable
taking value $k\in \{1,\dots,2^L\}$ with probability $p_k=|J_k|/m$, 
and create a random subsampling $X$ of $Y$ through
\be
X=\{x^i\in Y \, : \, i\in J_\kappa \}.
\ee
Then the budget condition $|X|=|J_\kappa|\leqslant C_0n$ is satisfied according to \iref{cardJ}.
Here, we define the discrete norm as
\be
\|v\|_X^2:=\frac{1}{|X|}\sum_{i\in J_\kappa}w(x^i)|v(x^i)|^2,
\ee
and the associated Gram matrix
\be
G_X:=(\langle L_j,L_k\rangle_X)_{j,k=1,\dots,n}=\frac m{|J_\kappa|}\sum_{i\in J_\kappa}a_ia_i^*.
\ee
The weighted least-squares estimate is now defined as
\be
\tilde u:= \underset{v\in V_n}{\argmin} \frac{1}{|X|}\sum_{i\in J_\kappa}w(x^i)|u(x^i)-v(x^i)|^2,
\ee
and it thus depends on the random draws of both $Y$ and $\kappa$.
Condition \iref{condition in V_n} follows from the lower inequality in \iref{stabJ} with $\beta=\frac{C_0}{c_0}$ since
\be
G_X\geqslant \frac m{|J_\kappa|}\frac n mc_0I\geqslant \frac{c_0}{C_0}I.
\ee
Finally, we have for any $v\in L^2(D,\mu)$
\be
\mathbb E_X(\|v\|_X^2)=\mathbb E_Y\left(\sum_{k=1}^{2^L}\frac{p_k}{|J_k|}\sum_{i\in J_k}w(x^i)|v(x^i)|^2\right)=\mathbb E_Y(\|v\|_Y^2)\leqslant 2e_n(u)^2,
\ee
so condition \iref{condition on expectancy} holds with $\alpha=2$. Applying Lemma \ref{main lemma},
we conclude that \iref{mainest} holds with $C=C_0$ and $K=1+2\frac {C_0}{c_0}$.
\hfill $\Box$

\section{Comparison with related results}

In order to compare Theorem \ref{main theorem} with several recent results \cite{KUV, KU, NSU, T2}, we 
consider its implication when the target function $u$ belongs to a certain class
of functions $\mathcal K$ that describes some prior information on $u$, such as smoothness.

Recall that if $V$ is a Banach space of functions defined on $D$ and $\mathcal K\subset V$ 
is a compact set, its {\it Kolmogorov $n$-width} is defined by
\be
d_n(\mathcal K)_{V}:=\inf_{\dim V_n=n}\sup_{u\in \mathcal K}\inf_{v\in V_n}\|u-v\|_{V},
\ee
where the first infimum is taken over all linear spaces $V_n\subset V$ of dimension $n$.
This quantity thus describes the best approximation error that can be achieved uniformly over
the class $\mathcal K$ by an $n$-dimensional linear space.

On the other hand, building a best approximation of $u$ requires in principle full knowledge
on $u$, and we want to consider the situation where we only have access to a limited number
of point evaluations. This leads one to consider the {\it sampling numbers}, also called {\it optimal recovery numbers}, both in 
the deterministic and randomized settings. 

For deterministic samplings, we define the (linear) sampling numbers
\be
\rho_m^{\det}(\mathcal K)_{V}:=\inf_{X, R_X}\max_{u\in \mathcal K}\|u-R_X(u(x^1),\dots,u(x^m))\|_{V},
\ee
where the infimum is taken over all samples $X=\{x^1,\dots,x^m\}\in D^m$
and linear reconstruction maps $R_X:\mathbb C^m\to V$. For random samplings, we may
define similar quantities by
\be
\rho_m^{\rm rand}(\mathcal K)_{V}^2:=\inf_{X,R_X}\max_{u\in \mathcal K}
\mathbb E\left(\|u-R_X(u(x^1),\dots,u(x^m))\|_{V}^2\right),
\ee
where the infimum is taken over all random variables $X=\{x^1,\dots,x^m\}\in D^m$
and linear reconstruction maps $R_X:\mathbb C^m\to V$. Note that a deterministic sample
can be viewed as a particular choice of random sample following a Dirac distribution in $D^m$,
and therefore
\be
\rho_m^{\rm rand}(\mathcal K)_{V}\leqslant \rho_m^{\rm det}(\mathcal K)_{V}.
\ee
Sampling numbers may also be defined without imposing the linearity of $R_X$,
leading to smaller quantities. In what follows, we shall establish upper bounds on the linear sampling numbers, which in turn are upper bounds for the nonlinear ones. We refer to \cite{NW1} for an introduction and study of sampling numbers in the context of general linear measurements, and to \cite{NW3} that focuses on point evaluation,
also termed as {\it standard information}.

By optimizing the choice of the space $V_n$ used in
Theorem \ref{main theorem}, we obtain as a consequence 
that, for $V=L^2=L^2(D,\mu)$, the sampling numbers in the randomized setting
are dominated by the Kolmogorov $n$-widths.

\begin{corollary}
For any compact set $\mathcal K\subset L^2$, one has
\be
\rho_{Cn}^{\rm rand}(\mathcal K)_{L^2}\leqslant K d_n(\mathcal K)_{L^2},
\label{randomized result}
\ee
where $C$ and $K$ are the same constants as in Theorem \ref{main theorem}.
\end{corollary}

\begin{remark}
The bound \iref{randomized result} cannot be attained with independent and identically distributed sampling points $x^1,\dots,x^m$. Indeed, consider the simple example, already evoked in \cite{Tr}, where $D=[0,1]$, $\mu$ is the Lebesgue measure,
\be
V_n=\left\{\sum_{i=1}^{n}a_i\chi_{\left[\frac {i-1} n, \frac {i} n\right[}, (a_1,\dots,a_n)\in \mathbb C^n\right\}
\ee
is a space of piecewise constant functions, and $\mathcal K=\{u\in V_n, \|u\|_{L^\infty}\leqslant1\}$. Then $\mathcal K\subset V_n$ so $d_n(\mathcal K)_{L^2}=0$, and an exact reconstruction $R_Xu=u$ is possible if and only if $X$ contains at least one point in each interval $\left[\frac {i-1} n, \frac {i} n\right[$. Thus $\rho_n^{\rm det}(\mathcal K)_{L^2}=0$, but in the case of i.i.d measurements, $m$ has to grow like $n\log n$ to ensure this constraint, due to the coupon collector's problem.
\end{remark}

\begin{remark}
In \cite{K}, a result similar to Theorem \ref{main theorem} is obtained
under the extra assumption of a uniform bound on $e_n(u)/e_{2n}(u)$,
yielding the validity of \iref{randomized result}
assuming a uniform bound on $d_n(\mathcal K)_{L^2}/d_{2n}(\mathcal K)_{L^2}$.
The recovery method used in \cite{K} is not of least-square type, but 
rather an elaboration of the pseudo-spectral approach that would simply approximate the 
inner products $\langle u,L_j\rangle=\int_D uL_jd\mu$ by a quadrature, using a hierarchical approach introduced in \cite{WW}.
\end{remark}

Ideally, one would like a ``worst case'' or ``uniform'' version of Theorem \ref{main theorem}, in the form
\be
\rho_{Cn}^{\det}(\mathcal K)_{L^2}\leqslant K d_n(\mathcal K)_{L^2},
\label{ideal}
\ee
but it is easily seen that such an estimate cannot be expected for general compact sets of
$L^2$, due to the fact that pointwise evaluations are not continuous in $L^2$ norm.

It is however possible to recover such uniform estimates by mitigating 
the non-achievable estimate \iref{ideal} in various ways.
One first approach, developed in \cite{LT, T2}, gives an inequality similar to \iref{ideal}, with $d_n(\mathcal K)_{L^2}$ replaced by $d_n(\mathcal K)_{L^\infty}$.
It is based on the following lemma, see Theorem 2.1 in \cite{T2}, which we recall for comparison with our Lemma \ref{main lemma}:
\begin{lemma}
Assume that $\mu$ is a finite measure of mass $\mu(D)=M<\infty$, that the constant functions belong to $V_n$, and that there exists a sample $X=\{x^1,\dots,x^m\}$ and weights $w^i$ such that the discrete norm
\be
\|v\|_X^2=\frac 1 {|X|}\sum_{i=1}^m w_i |v(x^i)|^2
\ee
satisfies a framing
\be
\beta^{-1}\|v\|^2\leqslant \|v\|_X^2 \leqslant \alpha\|v\|^2, \quad v\in V_n.
\label{temlyakov framing}
\ee
Then 
\be
\|u-P^X_{V_n}u\|\leqslant \sqrt M \left(1+\sqrt{\alpha\beta}\right)e_n(u)_{L^\infty},
\ee
where $e_n(u)_{L^\infty}=\min_{v\in V_n}\|u-v\|_{L^\infty}$.
\end{lemma}

\begin{proof}
For any $v\in L^2$, we have $\|v\|^2\leqslant M\|v\|_{L^\infty}^2$, and as $1\in V_n$,
\be
\|v\|_X^2\leqslant  \|1\|_X^2 \|v\|_{L^\infty}^2\leqslant \alpha\|1\|^2\|v\|_{L^\infty}^2=\alpha M\|v\|_{L^\infty}^2.
\ee
Hence
\begin{align*}
\|u-P^X_{V_n}u\|&\leqslant \|u-v\|+\|v-P^X_{V_n}u\| \\
&\leqslant \|u-v\|+\sqrt\beta\|v-P^X_{V_n}u\|_X\\
&\leqslant \|u-v\|+\sqrt\beta \|v-u\|_X\\
&\leqslant (\sqrt M+\sqrt{\alpha\beta M})\|u-v\|_{L^\infty},
\end{align*}
and we conclude by optimizing over $v\in V_n$.
\end{proof}

Here, in contrast to the derivation of \iref{statement} in Lemma \ref{main lemma},
one only uses the framing property \iref{temlyakov framing}, and does not need the condition
$\mathbb E(\|v\|_X^2)\leqslant \alpha \|v\|^2$. For this reason, one may achieve the above
objective with a simpler sparsification approach proposed in \cite{BSS} and adapted in \cite{LT},
which performs a greedy selection of the points $x^i$ 
within the sample $Y$, together with the definition of weights $w_i$
associated with these points. If the initial sample $Y$ satisfies
\be
\frac 1 2 I\leqslant G_Y \leqslant \frac 3 2 I,
\ee
then, for any $c>1$ the selection algorithm produces a sample 
$X$ of at most $cn$ points such that
\iref{temlyakov framing} holds with $\alpha=\frac 3 2 \left(1+\frac 1 {\sqrt c}\right)^2$ and $\beta^{-1}= \frac 12\left(1-\frac 1 {\sqrt c}\right)^{2}$.

Optimizing the choice of $V_n$ (but imposing that constant functions are contained
in this space), this leads to the following comparison result 
between deterministic optimal recovery numbers in $L^2$ and $n$-widths in $L^\infty$: for any compact set $\mathcal K\in \mathcal C(D)$, one has
\be
\rho_{cn}^{\det}(\mathcal K)_{L^2}\leqslant C\sqrt M d_{n-1}(\mathcal K)_{L^\infty},
\label{temlyakov result}
\ee
where $C$ depends on $c>1$. For $c=2$, one can take $C=11$.
We refer to \cite{LT,T2} where this type of result is established.

Another approach consists in making pointwise evaluations continuous by restriction to the case where $\mathcal K=B_H$ is the unit ball of a reproducing kernel Hilbert space $H\subset L^2$,
and assuming that the sequence $(d_n(B_H)_{L^2})_{n\geqslant1}$ is $\ell^2$-summable. The following result from \cite{NSU}, also based on the sparsification techniques from
\cite{MSS}, improves on a bound found in \cite{KU}
\be
\rho_{Cn}^{\det}(B_H)^2_{L^2}\leqslant K\frac{\log n}{n}\sum_{k\geqslant n}d_k(B_H)_{L^2}^2,
\label{NSU result}
\ee
More general compact
classes $\mathcal K$ of $L^2$, such that point evaluations are well defined on functions of $\mathcal K$,
are considered in \cite{KU2}, where the following general result is established: if
\be
d_n(\mathcal K)_{L^2}^2\leqslant C n^{-\alpha} \ln(n+1)^\beta, \quad n\geqslant 0,
\ee
for some $\alpha>1$ and $\beta\in\mathbb R$, then
\be
\rho_{n}^{\det}(\mathcal K)_{L^2}^2\leqslant \tilde C n^{-\alpha} \ln(n+1)^{\beta+1}, \quad n\geqslant 0.
\label{KU2 result}
\ee
In the above results, the additional logarithmic factor appears as a residual of the result obtained before sparsification, contrarily to the bounds \iref{randomized result} and \iref{temlyakov result}, which do not explicitely depend on the size of the initial sample $Y$. This results in a gap of
a factor $\log n$ between \iref{NSU result}, \iref{KU2 result} and known lower bounds for $\rho_{n}^{\det}(\mathcal K)_{L^2}^2$, see \cite{NSU}.

\section{Computational aspects}

The various results \iref{randomized result}, \iref{temlyakov result}, \iref{NSU result}, \iref{KU2 result}
ensure the existence of good sampling and
reconstructions algorithms in various settings. We end by a discussion on
the computational cost of these strategies. 

For the weighted least-squares methods corresponding to samples $Z$ and $Y$, the most expensive step consists in assembling the matrix $G_Z$ as a sum of $m$ matrices of size $n$, so the algorithmic complexity is of order $\mathcal O(mn^2)=\mathcal O(n^3\log n)$. Besides, to obtain $G_Y$, this step may need to be repeated a few times, as explained in Remark~\ref{remark redraws}, but this only affects the offline complexity by a small random factor.

Note that we assumed that an orthogonal basis $(L_1,\dots,L_n)$ of $V_n$ is explicitly known, which might not be the case for irregular domains $D$. However, under reasonable assumptions on $D$ or $V_n$, one can compute an approximately orthogonal basis $(\widetilde L_1,\dots,\widetilde L_n)$, either by performing a first discretization of $D$ with a large number of points, or by using a hierarchical method on a sequence of nested spaces $V_1\subset\dots\subset V_n$, see \cite{AC, ABC, CD, HNP, Mig2, Mig}. These additional steps have complexities $\mathcal O(K_n n^2)$ and $\mathcal O(n^4)$ respectively, where $K_n$ 
is the maximal value of the inverse Christoffel function $\sum_{j=1}^n |L_j|^2$
which might grow more than linearly with $n$ for certain choices of spaces $V_n$.
Results similar to Lemma \ref{boosted} have been obtained in the above references, with $(L_j)_{j=1,\dots,n}$ replaced by $(\widetilde L_j)_{j=1,\dots,n}$.

One could stop at this point and compute the approximation $\tilde u=P^Y_{V_n}u$, which satisfies error bounds both in expectation when comparing to $e_n(u)$, see Lemma \ref{boosted}, or uniformly when comparing to $e_n(u)_{L^\infty}$, see Theorem 1 (iii) in \cite{CM}. 
 Once the measurements of $u$ are performed, the computation of $\tilde u$ requires to solve a $n\times n$ linear system as in Remark \ref{remgram}, so the online stage takes a time $\mathcal O(\tau n\log n+n^3)$, where $\tau$ is the cost of each measurement of $u$.
 
However, in applications where the evaluation cost $\tau$ 
becomes very high (for example when each evaluation $x\mapsto u(x)$ requires
solving a PDE by some numerical code, or running a physical experiment), 
further reduction of the size of the sample may prove interesting, and justifies the interest for sparsification methods. The greedy selection method from \cite{BSS}, which is used in \cite{T2} and leads to \iref{temlyakov result}, has a complexity in $\mathcal O(mn^3)=\mathcal O(n^4\log n)$, but it can only be applied to the worst-case setting, with the uniform error bound $e_n(u)_{L^\infty}$. 

On the other hand, the iterative splitting method that we have used in this paper
following the ideas from \cite{MSS,NSU} is not easily implemented,
and one obvious method consists in testing all partitions of $\{1,\dots,m\}$ into sets $S_1$ and $S_2$ when applying Lemma \ref{MSS}. Note that this lemma is in practice used $L$ times, with $L=\mathcal O(\log\log n)$ since $2^L=\mathcal O(\frac m n)=\mathcal O(\log n)$. The algorithm consisting in subdividing the sample $L$ times, each time checking that the Gram matrices corresponding to $S_1$ and $S_2$ are well conditioned, and keeping one such subset at random, thus has an exponential complexity $\mathcal O(2^mn^3)=\mathcal O(n^{c n})$. Having
a different strategy that would produce the random sample in polynomial time is
currently an open problem to us. Note that 
the hierarchical Monte-Carlo approaches from \cite{WW,K}
have similar optimal error bounds with an optimal sampling budget, and without 
exponential complexity in the generation of samples, however
under the additional assumption that is described in Remark 4.

We summarise these computational observations in the following table, which illustrates the conflicts between reducing the sampling budget, ensuring optimal approximation results, and maintaining a reasonable cost
for sample generation.
\newline

\noindent
\centerline
{\begin{tabular}{|c|c|c|c|c|}
\hline
\begin{tabular}{@{}c@{}}sampling \\ complexity\end{tabular} &
\begin{tabular}{@{}c@{}}sample \\ cardinality $m$ \end{tabular} &
\begin{tabular}{@{}c@{}}offline \\ complexity \end{tabular} &
\begin{tabular}{@{}c@{}}$\mathbb E(\|u-\tilde u\|^2)$ \\ $\leqslant C e_n(u)^2$ \end{tabular} &
\begin{tabular}{@{}c@{}}$\|u-\tilde u\|^2$ \\ $\leqslant Ce_n(u)_{L^\infty}^2$ \end{tabular} \\
\hline
\begin{tabular}{@{}c@{}}conditioned \\ $\rho^{\otimes m}\,|\,E$\end{tabular} & $10n\log(4n)$ & $\mathcal O(n^3\log n)$ & \cmark & \cmark \\
\hline
\begin{tabular}{@{}c@{}} $+$ deterministic\\ sparsification \cite{BSS} \end{tabular} & $(1+\varepsilon)n$ & $\mathcal O(n^4\log n)$ & \xmark & \cmark \\
\hline
\begin{tabular}{@{}c@{}} $+$ random\\ sparsification \cite{MSS} \end{tabular} & $Cn$& $\mathcal O(n^{c n}) \to \mathcal O(n^r)$ ? & \cmark & \cmark\\
\hline
\end{tabular}
}
\newline
\newline

As a final remark, let us to emphasize that although the results presented in our paper
are mainly theorical and not practically satisfactory, due both to the computational complexity of the sparsification, and to the high values of the numerical constants $C$ and $K$ in Theorem \ref{main theorem}, they provide some intuitive justification to the boosted least-squares methods presented in \cite{HNP}, which consist in removing points from the initial sample as long as the corresponding Gram matrix $G_X$ remains well conditioned. For instance, Lemma \ref{MSS} allows to keep splitting the sample even after $L$ steps, if one still has a framing $\frac{1}{2}I\leqslant G_X\leqslant \frac 3 2 I$ and a sufficiently large ratio $\frac{|X|}{n}$. Nevertheless, it would be of much interest to find a randomized version of \cite{BSS} giving a bound of the form \iref{randomized result}, since this would give algorithmic tractability, smaller values for $C$ and $K$, and the possibility to balance these constants in Theorem \ref{main theorem}.


\end{document}